\documentclass[a4paper]{amsart}

\usepackage{graphics,epic}
\usepackage{amsmath,amssymb, amsthm,mathrsfs}
\usepackage[all,2cell]{xy}

\newtheorem{theorem}{Theorem}[section]
\newtheorem*{theorem*}{Theorem}

\newtheorem{lemma}[theorem]{Lemma}
\newtheorem{proposition}[theorem]{Proposition}
\newtheorem{corollary}[theorem]{Corollary}

\newtheorem*{conjecture*}{Conjecture}

\newtheorem{example}[theorem]{Example}
\newtheorem{remark}[theorem]{Remark}
\newtheorem{definition}[theorem]{Definition}

\newcommand{\opname}[1]{\operatorname{\mathsf{#1}}}

\newcommand{\Mod}{\opname{Mod}\nolimits}

\newcommand{\add}{\opname{add}\nolimits}

\newcommand{\rep}{\opname{rep}\nolimits}
\newcommand{\Rep}{\opname{Rep}\nolimits}

\newcommand{\im}{\opname{im}\nolimits}
\renewcommand{\ker}{\opname{ker}\nolimits}

\renewcommand{\P}{\mathbb{P}}

\newcommand{\ra}{\rightarrow}

\newcommand{\sgn}{\opname{sgn}}

\newcommand{\Hom}{\opname{Hom}}

\newcommand{\RHom}{\opname{RHom}}

\newcommand{\Ext}{\opname{Ext}}

\newcommand{\End}{\opname{End}}

\newcommand{\ten}{\otimes}
\newcommand{\lten}{\overset{\boldmath{L}}{\ten}}

\newcommand{\Tor}{\opname{Tor}}

\begin{document}

\title[Tilting theory of generalized preprojective algebras]{Tilting modules and support $\tau$-tilting modules over preprojective algebras associated with symmetrizable  Cartan matrices}\thanks{Partially  supported by  the China Scholarship Council and  the National Natural Science Foundation of China (No. 11471224) }

\author{Changjian Fu}
\address{Changjian Fu\\Department of Mathematics\\SiChuan University\\610064 Chengdu\\P.R.China}
\email{changjianfu@scu.edu.cn}
\author{Shengfei Geng}
\address{Shengfei Geng\\Department of Mathematics\\SiChuan University\\610064 Chengdu\\P.R.China}
\email{genshengfei@scu.edu.cn}
\maketitle

\begin{abstract}
For any given  symmetrizable  Cartan matrix $C$ with a symmetrizer  $D$, Gei\ss~ et al. (2016) introduced a generalized preprojective algebra $\Pi(C, D)$.  We study tilting modules and support $\tau$-tilting modules for  the generalized preprojective algebra $\Pi(C, D)$ and  show that there is a bijection between the set of all cofinite tilting ideals  of  $\Pi(C,D)$ and the corresponding Weyl group $W(C)$ provided that $C$ has no component  of Dynkin type. When $C$ is of Dynkin type,  we also establish a bijection between the set of all basic support $\tau$-tilting $\Pi(C,D)$-modules and  the corresponding Weyl group $W(C)$.  These results  generalize the classification results of Buan et al. (Compos. Math. 145(4), 1035-1079, 2009) and Mizuno 
(Math. Zeit. 277(3), 665-690, 2014) over classical preprojective algebras.
\end{abstract}

\section{Introduction}
Preprojective algebra associated to a quiver was invented by Gelfand and Ponomarev~\cite{GP}. Since its appearance, it has been studied extensively due to its relevance to various different parts of mathematics, see ~\cite{CB, L, N, R, GLS05} for instance. In particular, it has played a key role in Lusztig's construction~\cite{L} of semicanonical bases for the enveloping algebra $U(\mathfrak{n})$, where $\mathfrak{n}$ is a maximal nilpotent subalgebra of a complex symmetric Kac-Moody Lie algbra. The representation theory of preprojective algebras was also the foundation of Nakajima's construction of quiver varieties~\cite{N}.  Building on the work of Buan et al.~\cite{BIRS} and Gei\ss~et al.~\cite{GLS06,GLS11}, preprojective algebras also provided us a large class of  $2$-Calabi-Yau categories which lead to categorifications of certain important cluster algebras.

Recently, Gei\ss~et al.~\cite{GLS14} introduced a class of Iwanaga-Gorenstein algebras  via quivers with relations for any symmetrizable Cartan matrices with symmetrizers, which generalizes the path algebras of quivers associated with symmetric Cartan matrices. They also introduced the corresponding generalized preprojective algebras.  This new class of preprojective algebras reduces to the classical one provided that the Cartan matrix is symmetric and the symmetrizer is the identity matrix. Surprisingly, the generalized preprojective algebras still share many  properties with the classical one.  Since the classical preprojective algebras have many important applications in  different fields of mathematics, it is an interesting question  to find out which results or constructions for classical preprojective algebras can be generalized to the general setting. For example, if one can generalize the constructions of~\cite{BIRS, GLS11} to the new preprojective algebras, then one may obtain new categorifications for certain skew-symmetrizable cluster algebras.  

This note gives a first attempt to generalize certain classification results in tilting theory of preprojective algebras to this new setting.
For a given algebra, a basic question in tilting theory  is to  classify all the tilting modules or support $\tau$-tilting modules. For the classical preprojective algebras, the classification has been obtained by Buan~et al.~\cite{BIRS} for preprojective algebras of non-Dynkin type ({\it cf.} also ~\cite{IR}) and by Mizuno~\cite{M} for preprojective algebras of Dynkin type. We show that the classification results  of ~\cite{BIRS} and ~\cite{M} over classical preprojective algebras do generalize to  the generalized 
preprojective algebras  in the sense of ~\cite{GLS14}. More precisely, let $C$ be a  symmetrizable  Cartan matrix and $D$  a symmetrizer of $C$. Denoted by $\Pi=\Pi(C,D)$ the preprojective algebra associated to $(C, D)$. Denote by $\overline{Q}
$ the  quiver of $\Pi$ with
vertice set $Q_0 := \{1, \cdots , n\}$. Let $e_1, \cdots, e_n$  be the idempotents associated to the vertices. For each $i\in Q_0,$ denote by $I_i$ the two-sided ideal $\Pi(1-e_i)\Pi$ and
 $\langle I_1,I_2,\cdots,I_n\rangle$  the ideal semigroup generated by $I_1,I_2,\cdots,I_n.$  Denote by $W(C)$  the Weyl group of the symmetrizable Cartan matrix $C$.
We have the following main results.
\begin{theorem}
Let $C$ be a symmetrizable Cartan matrix with a symmetrizer $D$ and $\Pi=\Pi(C,D)$ the associated preprojective algebra. Assume that  $C$ has no component of Dynkin type,  then we have  bijections between the following sets:
\begin{enumerate}
\item[$(1)$] the set of all  cofinite tilting $\Pi$-ideals;
\item[$(2)$] the ideal semigroup $ \langle I_1,I_2,\cdots,I_n\rangle$;
\item[$(3)$] the Weyl group $W(C)$.
\end{enumerate}
\end{theorem}

\begin{theorem}
Let $C$ be a symmetrizable Cartan matrix of Dynkin type with a symmetrizer $D$ and $\Pi=\Pi(C,D)$ the associated preprojective algebra,  then we have  bijections between the following sets:
\begin{itemize}
\item[$(1)$] the set of all basic support $\tau$-tilting $\Pi$-modules;
\item[$(2)$] the ideal semigroup $ \langle I_1,I_2,\cdots,I_n\rangle$;
\item[$(3)$] the Weyl group $W(C)$.
\end{itemize}
\end{theorem}
Let us mention here that in contrast to the classical cases, we have to pay attention to locally free modules and generalized simple modules in this general setting. After establishing certain necessary properties for locally free modules and generalized simple modules, most of the remain arguments essentially follow the one in~\cite{BIRS, M}.
Moreover, for the preprojective algebras of Dynkin type, both the Weyl group and the support $\tau$-tilting $\Pi$-modules admit poset structures. One may adapt the argument in ~\cite{M} to show that the bijection in Theorem~$1.2$ is compatible with the poset structures.

The paper is organized as follows. In Section~\ref{S:basic-notions}, we recall basic definitions and properties of the  preprojective algebras associated with symmetrizable  Cartan matrices following~\cite{GLS14}.  In Section~\ref{S:non-Dynkin-type}, we  establish the bijection between the set of all cofinite tilting ideals of the preprojective algebras $\Pi(C,D)$ and  the ideal semigroup  $\langle I_1,I_2,\cdots,I_n\rangle$ provided that $C$ has no component of Dynkin type.  Section~\ref{S:Weyl-group-and-ideal-semigroup} is devoted to  prove that there exists a bijection between the Weyl group $W(C)$ and the ideal semigroup $\langle I_1,I_2,\cdots,I_n\rangle$ of $\Pi(C,D)$ for any symmetrizable Cartan matrix  $C$ with a symmetrizer $D$. In Section~\ref{S:Dynkin-type}, we  study the support $\tau$-tilting $\Pi(C, D)$-modules  for Dynkin type symmetrizable Cartan matrix $C$ and establish the bijection between the set of all support $\tau$-tilting  $\Pi(C,D)$-modules   and the ideal semigroup  $\langle I_1,I_2,\cdots,I_n\rangle$. 

\noindent{\bf Notation}. Let $\Lambda$ be an algebra over a field $K$. Denote by $\mathbb{D}=\Hom_K(-, K)$ the usual duality.
By a module, we mean a right module unless stated otherwise. 
For a $\Lambda$-module $M$, denote by $|M|$ the number of non-isomorphic indecomposable direct summands of $M$. For an integer $l$, denote by $M^l$ the direct sum of $l$ copies of $M$.  Denote by $\operatorname{add} M$ the subcategory of $\Lambda$-modules consisting of modules which are finite direct sum of direct summands of $M$.

\section{ Preliminary}~\label{S:basic-notions}

Following~\cite{GLS14}, we recall basic definitions and properties of the  preprojective algebras associated with symmetrizable generalized Cartan matrices.  In this section, we denote by  $K$  an arbitrary field.
\subsection{Symmetrizable 
Cartan matrix and orientation}
\begin{definition}A matrix $C = (c_{ij}) \in M_n(\mathbb{Z})$ is a {\it symmetrizable (generalized)
Cartan matrix} if the following conditions are satisfied:

$(C_1)$~ $c_{ii} = 2$ for all $i$;

$(C_2)$~ $c_{ij} \leq 0$ for all $i \neq j$;

$(C_3)$ ~$c_{ij} \neq 0$ if and only if $c_{ji} \neq 0$.

$(C_4)$ ~There is a diagonal integer matrix $D = \opname{diag}(c_1,\cdots , c_n)$ with $c_i \geq 1$ for all $i$ such
that $DC$ is symmetric.

The matrix $D$ appearing in $(C_4)$ is called a {\it symmetrizer} of $C$ and is 
{\it minimal} if $c_1 +\cdots + c_n$ is minimal.
Denote
$g_{ij} := | \operatorname{gcd}(c_{ij} , c_{ji})|, f_{ij} := |c_{ij} |/g_{ij} .$
\end{definition}
An {\it orientation} of $C$ is a subset $\Omega$ of
$\{1, 2,\cdots , n\} \times \{1, 2, \cdots , n\} $ such that the followings hold:

$(A_1)$~ $\{(i, j), (j, i)\} \cap \Omega
 \neq \emptyset $ if and only if $c_{ij} < 0$;

$(A_2)$~ For each sequence $((i_1, i_2), (i_2, i_3), \cdots , (i_t, i_{t+1}))$ with $t \geq 1$ and $(i_s, i_{s+1}) \in \Omega$
 for
all $1 \leq s \leq t$, we have $i_1 \neq i_{t+1}$.

Given an orientation $\Omega$
 of $C$, let $Q := Q(C,\Omega
) := (Q_0,Q_1, s, t)$ be the quiver with the set of
vertices $Q_0 := \{1, \cdots , n\}$ and with the set of arrows
$Q_1 := \{a^{(g)}_{ij} : j \to i | (i, j) \in \Omega, 1 \leq g \leq g_{ij}\} \cup\{\varepsilon_i : i \to i | 1 \leq i \leq n\}.$
Thus we have $s(a^{(g)}_{ij} ) = j$ and $t(a^{(g)}_{ij} ) = i$ and $s(\varepsilon_i) = t(\varepsilon_i) = i$, where $s(a)$ and $t(a)$
denote the starting and terminal vertex of an arrow $a$, respectively. If $g_{ij} = 1$, we also
write $a_{ij}$ instead of $a^{(1)}_{ij} $. 

The quiver $Q$ is called a {\it quiver of  type $C$} and we say the generalized Cartan matrix $C$ is {\it connected} if $Q$ is connected.
Denote by $Q^o$ the quiver obtained from $Q$ by deleting all the loops of $Q$. By the condition  $(A_2)$, we know that $Q^o$ is an acyclic quiver. 

\subsection{Preprojective algebras associated to symmetrizable Cartan matrices}~\label{s:preprojectivealgebra}

Let $C$  be a symmetrizable Cartan matrix with a symmetrizer $D$. Given an orientation  $\Omega$ of $C$, 
the opposite orientation of  $\Omega$
 is defined as
$\Omega^{op} :=\{(j, i) | (i, j) \in\Omega
\}$.  Denote $\overline{\Omega}=\Omega\cup\Omega^{op}$.
For $(i, j)\in \overline{\Omega}$, 
set
\[\sgn (i,j)=\begin{cases}+1& if\ (i, j) \in \Omega; \\-1&if\ (i, j) \in \Omega^{op}.\end{cases}\]
Let $Q$ be the quiver defined by  $(C,\Omega)$. Let $\overline{Q} = \overline{Q}(C,\Omega)$ be the quiver obtained from $Q$ by adding a new arrow
$a^{(g)}_{ji} : i \to j$ for each arrow $a^{(g)}_{ij} : j \to i \ of\ Q^{\circ}.$

Let $\Omega(i,-)$ be the set $\{j|(i,j)\in\Omega\}.$ Similar one can define ${\Omega}(-,i)$, $\overline{\Omega}(-,i)$ and $\overline{\Omega}(i,-)$.
Now we can define an  algebra
$H := H(C,D,\Omega
) := KQ/I,$ 
where $KQ$ is the path algebra of $Q$, and $I$ is the ideal of $KQ$ defined by the following
relations:

$(H_1)$ For each $i\in Q_0$,  we have the nilpotency relation
$\varepsilon^{c_i}_i = 0.$

$(H_2)$ For each $(i, j)\in \Omega$
 and each $1 \leq g \leq g_{ij}$,  we have the commutativity relation
$$\varepsilon_i^{f_{ji}}a^{(g)}_{ij} =a^{(g)}_{ij}\varepsilon_j^{f_{ij}}.$$
\begin{definition}
The {\it  preprojective algebra} $\Pi:= \Pi(C,D,\overline{\Omega})$ associated to the symmetrizable Cartan matrix $C$  is the quotient algebra $K\overline{Q}/\overline{I}$ of the path algebra $K\overline{Q}$ by the ideal
 $\overline{I}$ generated by the following relations:

$(P_1)$ For each $i\in Q_0$, we have the nilpotency relation
$\varepsilon_i^{c_i} = 0.$

$(P_2)$ For each $(i, j) \in\overline{\Omega}$
 and each $1 \leq g \leq g_{ij}$, we have the commutativity relation
$$\varepsilon_i^{f_{ji}}a^{(g)}_{ij} =a^{(g)}_{ij}\varepsilon_j^{f_{ij}}.$$

$(P_3)$ For each $i$, we have the mesh relation
\[\sum_{j\in\overline{\Omega}(-,i)}\sum_{g=1}^{g_{ij}}\sum_{f=0}^{f_{ji}-1}\sgn (i,j)\varepsilon_i^{f}a^{(g)}_{ij} a^{(g)}_{ji}\varepsilon_i^{f_{ji}-1-f}=0.\]
\end{definition}

When the generalized Cartan matrix $C$ is symmetric and the symmetrizer $D$ is the identity matrix, this definition reduces to the one of the classical preprojective algebras for acyclic quivers. As the classical preprojective algebras, the name is also justified by the fact that as an $H$-module, 
 $\Pi\cong\bigoplus_{j\in Q_0}\bigoplus_{k\in \mathbb{N}}\tau^{-k}(P_j)$,
 where $P_j$ is the projective $H$-module corresponding the vertex $j$ and $\tau$ is the Auslander-Reiten translation~({\it cf.}~\cite[Thm. $1.7$]{GLS14}).
 It is also easy to see that if the symmetrizer $D$ is not the identity matrix,  the quiver $\overline{Q}$  has loops.

By the definition of  $\Pi$, it is clear that $\Pi$ does not depend on the orientation of $C$. Hence in the following, we simply write $\Pi=\Pi(C,D)$.

\begin{example}\label{eg1}
Let $C=\left(\begin{array}{cc}2&-1\\-1&2 \end{array}\right)$, $D=\opname{diag}(2,2)$, $\Omega={(1,2)}$. The preprojective algebra $\Pi=\Pi(C,D)$ is given by the quiver
\[\xymatrix{
\ar@(ul,dl)_{\varepsilon_1}1\ar[rr]<1mm>^{a_{21}}&&\ar@(ur,dr)^{\varepsilon_2
}2\ar[ll]^{a_{12}}}
\]
with relations $\varepsilon_1^{2}=0, \varepsilon_2^{2}=0, \varepsilon_1a_{12}=a_{12}\varepsilon_2,\varepsilon_2a_{21}=a_{21}\varepsilon_1, a_{12}a_{21}=a_{21}a_{12}=0.$
\end{example}

\begin{example}\label{eg2}
Let $C=\left(\begin{array}{cc}2&-1\\-2&2 \end{array}\right)$, $D=\opname{diag}(2,1)$, $\Omega={(1,2)}$. The preprojective algebra $\Pi=\Pi(C,D)$ is given by the quiver
\[\xymatrix{
\ar@(ul,dl)_{\varepsilon_1}1\ar[rr]<1mm>^{a_{21}}&&\ar@(ur,dr)^{\varepsilon_2}2\ar[ll]^{a_{12}}
}
\]
with relations $\varepsilon_1^{2}=0=\varepsilon_2, \varepsilon_1^2a_{12}=a_{12}\varepsilon_2,\varepsilon_2a_{21}=a_{21}\varepsilon_1^2, a_{12}a_{21}\varepsilon_1+\varepsilon_1a_{12}a_{21}=0, a_{21}a_{12}=0$.
Since $\varepsilon_1^{2}=0=\varepsilon_2$, we can delete the loop $\varepsilon_2$ and then  $\Pi(C,D)$ is given by the quiver
\[\xymatrix{
\ar@(ul,dl)_{\varepsilon_1}1\ar[rr]<1mm>^{a_{21}}&&\ar[ll]^{a_{12}}2
}
\]
with relations $\varepsilon_1^{2}=0,  a_{12}a_{21}\varepsilon_1+\varepsilon_1a_{12}a_{21}=0, a_{21}a_{12}=0$.
\end{example}

\subsection{The quadratic form}
Given a symmetrizable Cartan matrix $C$  with  a symmetrizer $D$,
 the quadratic form $q_C : \mathbb{Z}_n \to \mathbb{Z}$ of $C$ is defined as follows
\[q_C :=\sum\limits_{i=1}^nc_iX_i^2-\sum\limits_{i<j}c_i|c_{ij}|X_iX_j.\]
The symmetrizable Cartan matrix $C$ is of Dynkin  type if and only if $q_C$ is positive definite. 
It is well-known that  connected symmetrizable Cartan matrices  of Dynkin type can be classified by the Dynkin diagrams of type $A_n,B_n,C_n,D_n,E_6,E_7,E_8,F_4,G_2.$

\begin{lemma}\cite[Corollary. 12.7]{GLS14}\label{selfinjective}
Suppose that the Cartan matrix $C$ is of Dynkin type. Then the preprojective algebra $\Pi(C, D)$
is a selfinjective algebra for any symmetrizer $D$ of $C$.
\end{lemma}

\subsection{Locally free modules}
Let $e_1, \cdots, e_n$ be the idempotents  corresponding
to the vertices of  $\overline{Q}$. Let $H_i$ be the
truncated polynomial ring $K[\varepsilon_i]/(\varepsilon^{c_i}_i )$. 
A  $\Pi$-module $M$ is called {\it locally free}  if $Me_i$ is  a free $H_i$-module for each $i$. Similarly, we can define locally free $H$-modules and we refer to~\cite{GLS14} for the precisely definition.
It was shown in~\cite{GLS14} that the projective  $H$-modules and the projective $\Pi$-module are locally free.

Let $\Rep_{l.f.}(\Pi)$ be  the category
of all locally free  $\Pi$-modules and  $\rep_{l.f.}(\Pi)$ the category
of all finitely generated locally free $\Pi$-modules.  
The proof of Lemma $3.8$ in~\cite{GLS14} also implies the following.
\begin{lemma}\label{closed under extension}
 $\Rep_{l.f.}(\Pi)$ is closed under extensions, kernels of epimorphisms and cokernels of monomorphisms.
\end{lemma}
Note that each projective $\Pi$-module is locally free,  we clearly have
\begin{corollary}\label{lem of finite proj dim}
Let   $M$ be a $\Pi$-module.  If  $\opname{proj.dim} M< \infty$, then  $M$ is locally free.
\end{corollary}

For preprojective algebras of non-Dynkin types,  the locally free modules are characterized by the finiteness of projective dimension.
Namely, we have
\begin{lemma}\cite[Proposition 12.3]{GLS14}\label{nondynkin equivalent}
Let $C$ be a symmetrizable Cartan matrix with a symmetrizer $D$ which has no component of Dynkin type and $\Pi=\Pi(C, D)$ the associated preprojective algebra.
Let  $M$ be  a  $\Pi$-module, then the following statements are equivalent: 

${(1)}$ $M$ is locally free;  

${(2)}$ $\opname{proj.dim} M\leq 2$;

${(3)}$ $\opname{proj.dim} M< \infty$.
\end{lemma}

The following well-known fact for classical preprojective algebras has been generalized to locally free modules for the preprojective algebras associated with symmetrizable Cartan matrices in ~\cite{GLS14}.
\begin{lemma}\cite[Theorem 12.6]{GLS14}\label{2-cy}
Let $M \in \Rep_{l.f.}(\Pi),$ $ N \in \rep_{l.f.}(\Pi).$ Then

(a) there is a functorial isomorphism
\[\Ext^1_{\Pi}(M,N) \cong 
\mathbb{D}\Ext^1_{\Pi}(N,M).\]

(b) if   $C$ has no component of Dynkin type, there are
functorial isomorphisms
\[\Ext^{2-i}_{\Pi}(M,N)\cong
\mathbb{D}\Ext^i_{\Pi}(N,M)~\text{for~$i=0,1,2$.}\]
\end{lemma}

\subsection{Generalized simple modules $E_i$}

Let $C$ be a symmetrizable   Cartan matrix with a symmetrizer $D = \opname{diag}(c_1,\cdots , c_n)$.
 Let $M$ be a locally free  $\Pi$-module. For each $i \in Q_0$, let $r_i$ be the rank of the free $H_i$-module
$Me_i$.
 Put $\opname{rank}\nolimits(M) := (r_1, \cdots , r_n)$ and we  call $r_1+\cdots+r_n$ the {\it rank length} of $M$.

Let $E_1, \cdots ,E_n$ be the indecomposable locally free  $\Pi-$modules
with $\opname{rank}\nolimits(E_i) = \alpha_i$, where $\alpha_1,\cdots , \alpha_n$ is the standard basis of $\mathbb{Z}_n$.  In particular, $E_i$ is the regular representation of $H_i$ and we call $E_i$  {\it  the generalized simple} $\Pi$-module. 
 Obviously, $E_i$ is also an $H$-module and we also call it a generalized simple $H$-module. 
If we denote by $S_i$  the simple $\Pi$-module associated to the vertex $i$, then $E_i$ is the uniserial module with composition factor $S_i$ of  length $c_i$. For each $1\leq d\leq c_i$, let  $^{d}S_i$  be the uniserial module of  length $d$.

\begin{proposition}
 For each $j\in Q_0$, denote by $e_j\Pi$ the projective right $\Pi$-module associated to the vertex $j$. We have
 \begin{enumerate}
 \item[(a)]~ $\Hom_{\Pi}(e_i\Pi,E_i)\cong E_i$ and $\Hom_{\Pi}(e_j\Pi,E_i)=0$ for $i\neq j$, where $E_i$ is the generalized simple $\Pi$-module associated to $i$;
 \item[(b)]~ $e_i\Pi\otimes_{\Pi}E_i\cong E_i$ and  $e_j\Pi\otimes_{\Pi}E_i=0$ for $i\neq j$,  where $E_i$ is the generalized simple $\Pi^{op}$-module associated to $i$.
 \end{enumerate}
 \end{proposition}
 \begin{proof}
 Note that $\Hom_{\Pi}(e_j\Pi, E_i)\cong E_ie_j$ and $e_j\Pi\ten_{\Pi}E_i\cong e_jE_i$. Now the results follow from the fact that $e_jE_i=0=E_ie_j$ for $i\neq j$ and $e_iE_i=E_i=E_ie_i$.
 \end{proof}

We know that if $C$ is of Dynkin type, by Lemma  \ref{selfinjective},  $\Pi$ is a selfinjective algebra.
Denote by $\nu=\mathbb{D}\Hom_{\Pi}(-,\Pi)$   the Nakayama functor and $\sigma:Q_0\to Q_0$  the Nakayama permutation of $\Pi$, $i.e.\ e_i\Pi\cong \mathbb{D}(\Pi
e_{\sigma(i)})$. Then we have
\begin{proposition}\label{ci=csigmai}
Let $C$ be a symmetrizable Cartan matrix of Dynkin type with a symmetrizer $D$ and  $\Pi=\Pi(C,D)$ the associated preprojective algebra. 
We have $\nu E_{\sigma(i)}\cong E_i$ and $c_i=c_{\sigma({i})}$ for each $i\in Q_0$.
\end{proposition}
\begin{proof}
We show that for each $1\leq d\leq c_{\sigma(i)}$, $\nu (~^dS_{\sigma(i)})\cong ~^dS_i$ by induction on $d$. For $d=1$, we clearly have $\nu(S_{\sigma(i)})=\nu(\opname{soc}e_i\Pi)=\opname{top}e_i\Pi=S_i$. Now suppose that we have $\nu(~^{d-1}S_{\sigma(i)})\cong~^{d-1}S_i$. Consider the following  non-split short exact sequence
\[0\to S_{\sigma(i)}\to ~^dS_{\sigma(i)}\to ~^{d-1}S_{\sigma(i)}\to 0.
\]
Recall that $\Pi$ is selfinjective and hence the functor $\nu$ is exact. Applying $\nu$ to the above short exact sequence yields the following non-split short exact sequence
\[0\to \nu(S_{\sigma(i)})\to \nu(~^dS_{\sigma(i)})\to \nu(~^{d-1}S_{\sigma(i)})\to 0.
\]
Now one can deduce $\nu(~^dS_{\sigma(i)})\cong ~^dS_i$ from the fact that there is a unique non-split extension of $S_i$ by $~^{d-1}S_i$ whose middle term is $~^dS_i$. In particular, we have $\nu(E_{\sigma(i)})\cong~^{c_{\sigma(i)}}S_i$ which is indecomposable. Consequently, $c_{\sigma(i)}\leq c_i$. Similarly, using the functor $\nu^{-1}=\Hom_{\Pi}(\mathbb{D}\Pi,-)$, one can show that $c_i\leq c_{\sigma(i)}$. Hence we have $c_i=c_{\sigma(i)}$ and $\nu(E_{\sigma(i)})\cong E_i$.
\end{proof}

The following result is an analogue of the classical cases~({\it cf.}~\cite[Prop 4.2]{BBK}) which plays an important role in our investigation.
\begin{lemma}~\label{lem of EI resolution}
Let $\Pi=\Pi(C,D)$ be a preprojective algebra associated to a symmetrizable Cartan matrix $C$ with a symmetrizer $D$. Let $E_i$ be a generalized simple $\Pi$-module.

If $C$ has no component of Dynkin type, then we  have an exact sequence
\begin{eqnarray}0\to e_i\Pi\to \bigoplus_{j\in \overline{\Omega}(i,-)}(e_j\Pi)^{|c_{ji}|}\to e_i\Pi\to E_i\to 0.\label{resolution of E_I nondynkin case}\end{eqnarray}

If  $C$ is of Dynkin type, then we  have an exact sequence
\begin{eqnarray}0\to E_{\sigma(i)}\to e_i\Pi\to \bigoplus_{j\in  \overline{\Omega}(i,-)}(e_j\Pi)^{|c_{ji}|}\to e_i\Pi\to E_i\to 0,\label{resolution of E_I dynkin case}\end{eqnarray}
where $\sigma:Q_0\to Q_0$ is a Nakayama permutation of $\Pi$. 

\end{lemma}
\begin{proof}
Note that the exact sequence (\ref{resolution of E_I nondynkin case}) is a direct consequence of ~\cite[Prop 12.1]{GLS14}. Namely, by applying the functor $E_i\ten_{\Pi}$ to the sequence of Proposition ~$12.1$ in ~\cite{GLS14}, we obtain the following sequence
\[e_i\Pi\xrightarrow{f} \bigoplus_{j\in\overline{\Omega}(i,-)}(e_j\Pi)^{|c_{ji}|}\to e_i\Pi\to E_i\to 0,
\]
which is the beginning of a projective resolution of $E_i$. 

If $C$ has no component of Dynkin type, then by Lemma~\ref{nondynkin equivalent}, we deduce that the map $f$ is injective and hence obtain the exact sequence (\ref{resolution of E_I nondynkin case}).

Now assume that $C$ is of Dynkin type. In this case, the preprojective algebra $\Pi$ is a finite-dimensional selfinjective algebra over $K$.
Denote by $M$ the kernel of $f$. It remains to show that $M\cong E_{\sigma(i)}$. Note that one can easily see that $M$ is locally free by Lemma~\ref{closed under extension}. On the other hand, we clearly have $\opname{soc} M=\opname{soc} e_i\Pi=S_{\sigma(i)}$. Thus, to show $M\cong E_{\sigma(i)}$, it suffices to prove that $\dim_K M=c_{\sigma(i)}$.

Now choose an orientation $\Omega$ of $C$
 such that $i$ is a sink vertex in the quiver $Q^o$. Let $H = H(C,D,\Omega
)$ be the algebra defined in Section~\ref{s:preprojectivealgebra}. Denote by $P_j=He_j$ the left projective $H$-module corresponding to the vertex $j$. Since $i$ is a sink vertex in $Q^o$, one has $\dim_K\Hom_{H}(P_i,P_i)=c_i$. Consider the following short exact sequence of left $H$-modules
\[0\to P_i\to \bigoplus_{j\in \Omega(i,-)}P_j^{|c_{ji}|}\to \tau^{-1}P_i
\to 0,\]
where $\tau$ is the Auslander-Reiten translation of left $H$-modules. Applying the functor $\Hom_H(-, \Pi)$ to the above exact sequence, we obtain an exact sequence of right $\Pi$-modules
\[0\to \Hom_H(\tau^{-1}P_i,\Pi)\to \bigoplus_{j\in \Omega(i,-)}(e_j\Pi)^{|c_{ji}|}\to e_i\Pi\to \Ext_H^1(\tau^{-1}P_i,\Pi)\to 0.\]
Recall that  as a left $H$-module, we have
\[\Pi=\bigoplus_{j\in Q_0}\bigoplus_{k\in \mathbb{N}}\tau^{-k}(P_j),
\]
which implies that $\Ext_H^1(\tau^{-1}P_i,\Pi)\cong \Ext_H^1(\tau^{-1}P_i,P_i)\cong \mathbb{D}\Hom_{H}(P_i,P_i)$. In particular, $\dim_K\Ext_H^1(\tau^{-1}P_i,\Pi)=c_i=\dim_K E_i$.
Finally, consider the split exact sequence of left $H$-modules
\[0\to \mathbb{D}H\to \Pi\to \tau \Pi\to 0,
\]
where the first map is an inclusion and the second one is the projection of $\Pi$ onto its
direct summand $\tau \Pi$. Applying the functor $\Hom_H(P_i,-)$, we obtain a short exact sequence 
\[0\to \Hom_H(P_i,\mathbb{D}H) \to \Hom_H(P_i,\Pi) \to \Hom_H(P_i,\tau \Pi)\to 0.
\]
It is clear that $\Hom_H(P_i,\tau \Pi)\cong\Hom_H(\tau^{-1}P_i, \Pi)$ and $\Hom_H(P_i,\Pi)\cong e_i\Pi$. Note that $i$ is a sink vertex in $Q^o$, which implies that $\overline{\Omega}(i,-)=\Omega(i,-)$. We conclude that $\dim_K M=\dim_K\Hom_H(P_i, \mathbb{D}H)=\dim_K\Hom_H(H, P_i)=c_i=c_{\sigma(i)}$ by Proposition~\ref{ci=csigmai}.
\end{proof}

\subsection{ The two-sided  ideal $I_i$}

Let $\Pi=\Pi(C,D)$ be a preprojective algebra. For each $i\in Q_0,$ denote by $I_i$ the two-sided ideal $\Pi(1-e_i)\Pi$,  it is easy to see that $e_jI_i=e_j\Pi(1-e_i)\Pi=e_j\Pi$ for $j\neq i$.
Thus we obtain the following decomposition of $I_i$ as right $\Pi$-module 
\[I_i= \bigoplus_{j\in Q_0}e_jI_i= e_iI_i\oplus (\bigoplus_{j\neq i}e_j\Pi).
\] 
By the definition of $I_i$, we clearly have the following short exact sequence
 \begin{eqnarray}0\to I_i\to \Pi\to E_i\to 0,\label{exact sequcence Ii and ei}\end{eqnarray}
 which induces a short exact sequence
 \begin{eqnarray}0\to e_iI_i\to e_i\Pi\to E_i\to 0.\label{exact sequcence eIi and ei}\end{eqnarray}
Now by  Lemma \ref{lem of EI resolution}, we obtain a projective presentation of $e_iI_i$: \begin{eqnarray}e_i\Pi\to \bigoplus_{j\in \overline{\Omega}(i,-)}(e_j\Pi)^{|c_{ji}|}\to e_iI_i\to 0.\notag\end{eqnarray}
Since $e_i\Pi$, $\Pi$, $E_i$ are locally free, by Lemma \ref{closed under extension}, both $I_i$ and $e_iI_i$ are locally free.
On the other hand,  $\Hom_{\Pi}(e_j\Pi, E_i)=0$ for $j\neq i,$  we obtain that $\Hom_{\Pi}(e_iI_i, E_i)=0$. In particular, we have proved the following result.
\begin{proposition}For any $i\in Q_0,$ $I_i$ and $e_iI_i$ are locally free and $ \Hom_{\Pi}(I_i, E_i) = 0$.
\end{proposition}

\section{Preprojective algebras of non-Dynkin type }~\label{S:non-Dynkin-type}

From this section to the end of this paper, we assume moreover that $K$ is an algebraically closed field.
The purpose of this section is to  generalized the classification results of \cite{BIRS} over classical preprojective algebras of non-Dynkin types. Hence, in this section, we always assume that $C$ is a symmetrizable Cartan matrix with a symmetrizer $D$ which has no component of Dynkin type and $\Pi=\Pi(C, D)$ the associated preprojective algebra. We follow~\cite[Chapter II]{BIRS}.

\begin{definition}
For an algebra $\Lambda$, we say that a finitely presented $\Lambda$-module $T$ is a {\it tilting module} if:
\begin{itemize}
\item[$(i)$]~ there exists an
exact sequence $0\to P_n \to \cdots \to P_0\to T \to 0 $ with finitely generated projective $\Lambda$-modules $P_i$;
\item [$(ii)$]~ $\Ext^i_{\Lambda}(T, T) = 0$ for any $i > 0$;
\item [$(iii)$]~ there exists an exact sequence $0\to\Lambda \to T_0 \to\cdots \to T_n \to 0$
with $T_i$ in $\add T$.
\end{itemize}

 A {\it partial $\Lambda$-tilting module} is a direct summand of a $\Lambda$-tilting module.
\end{definition}

Since tilting modules have finite projective dimension,   the projective dimension of a partial tilting module is also finite. The following result is a direct consequence of  Lemma \ref{nondynkin equivalent}.
\begin{proposition}\label{tilting are free}

All tilting  $\Pi$-modules and partial tilting $\Pi-$modules are locally free and have projective dimension at most two.
\end{proposition}

For a partial tilting module  of projective dimension at most one, we  have
\begin{lemma}\label{lem of T times E}
 Let $T$ be a partial tilting $\Pi$-module of projective dimension at most one and $E_i$  a generalized simple $\Pi^{op}$-module associated to the vertex $i$. Then at least one of the statements $T
\otimes_{\Pi} E_i = 0$ and $\operatorname{\mathrm{Tor}}\nolimits^{\Pi}_1(T,E_i) = 0$
holds.
\end{lemma}
\begin{proof}
Let $0\to P_1\to P_0 \to T \to 0 $ be a  minimal projective resolution of $T$. Applying $-\otimes _{\Pi}E_i$ to this sequence, we get an  exact sequence
$$0\to \Tor_{\Pi}^1(T, E_i)\to P_1 \otimes_{\Pi}E_i \to P_0 \otimes_{\Pi}E_i\to T \otimes_{\Pi}E_i \to 0 .$$
Since $P_0$ and $P_1$ do not have a common direct summand  and $e_j\Pi\otimes_{\Pi}E_i=0$ when $j\neq i$,  we have either  $P_1 \otimes_{\Pi}E_i=0$ or $ P_0 \otimes_{\Pi}E_i=0$, which implies the desired result.
\end{proof}

The following result on (partial) tilting complexes are useful in studying derived equivalences. For  the definitions of tilting complex and two-sided tilting complex, we refer to ~\cite{Ric} for details.
\begin{lemma}\cite{Ric,Yek}\label{key lem}
For rings $\Lambda$ and $\Gamma$, let $T\in D(\Mod \Lambda\otimes_{\mathbb{Z}}\Gamma^{op})$ be a two-sided tilting complex. 

$(a)$ For any tilting complex (respectively, partial tilting complex) $U$ of $\Gamma$, we have a tilting
complex (respectively, partial tilting complex) $T\lten_{\Gamma}U$ of $\Lambda$ such that $\End_{D(\Mod \Lambda)}(T\lten_{\Gamma}U)\cong \End_{D(\Mod \Gamma)}(U)$;

$(b)$ $ \RHom_{\Lambda}(T, \Lambda)$ and $ \RHom_{\Gamma^{op}}(T,\Gamma) $ are two-sided tilting complexes which are isomorphic in
$D(\Mod~\Gamma\otimes_{\mathbb{Z}}\Lambda^{op})$.
\end{lemma}
We also remark that  a tilting module is nothing but a module which is a tilting complex.

\begin{proposition}\label{Ii is tilting}
 The two-side ideal $I_i$ is a tilting $\Pi$-module of projective dimension at most one and
$\End_{\Pi}(I_i)\cong\Pi$.
\end{proposition}
\begin{proof}
Recall that $E_i\in \rep_{l.f.}(\Pi)$ and $I_i\in \Rep_{l.f.}(\Pi)$, one deduces that \[\Ext^1_{\Pi}(I_i, I_i)\cong \Ext^2_{\Pi}(E_i, I_i)\cong \mathbb{D}\Hom_{\Pi}(I_i, E_i) = 0\]
by  the exact sequence (\ref{exact sequcence Ii and ei}) and Lemma \ref{2-cy}. Moreover, we have \[\Ext^{k}_{\Pi}(E_i,\Pi)\cong
\mathbb{D}\Ext^{2-k}_{\Pi}(\Pi,E_i)=0\]
 for $k = 0, 1$. Now applying the functor $\Hom_{\Pi}( -, \Pi)$ to the exact sequence (\ref{exact sequcence Ii and ei}), we obtain $\Hom_{\Pi}(I_i, \Pi)\cong\Pi$.
Recall that $\Hom_{\Pi}(I_i, E_i) = 0$, we obtain $\End_{\Pi}(I_i)\cong \Pi$ by applying the functor $\Hom_{\Pi}(I_i,- )$ to the exact sequence (\ref{exact sequcence Ii and ei}).

It remains to show that $I_i$ is a tilting $\Pi$-module of projective dimension at most one.
Note that $I_i=e_iI_i\oplus(\bigoplus_{j\neq i}e_j\Pi)$ as right $\Pi$-module.
By the exact sequence (\ref{resolution of E_I nondynkin case}), we have the following projective resolution of $e_iI_i$:  \[0\to e_i\Pi\to \bigoplus_{j\in\overline{\Omega}(i,-)}(e_j\Pi)^{|c_{ji}|}\to e_iI_i\to 0.\]
In particular, it implies that $\opname{proj. dim} I_i\leq 1$. On the other hand, we also obtain the following exact sequence
\[0\to\Pi \to (\bigoplus_{j\in\overline{\Omega}(i,-)}(e_j\Pi)^{|c_{ji}|})\oplus (1-e_i)\Pi\to e_iI_i\to 0,
\]
which implies that $I_i$ is a tilting $\Pi$-module of projective dimension at most one.
\end{proof}

\begin{proposition}\label{chengji is tilting}
Let $T$ be a tilting $\Pi$-module of projective dimension at most one and $E_i$ the generalized simple $\Pi^{op}$-module associated to vertex $i$.

$(a)$ If $T\otimes_{\Pi} E_i = 0 ,$ then $T I_i=T$.

$(b)$ If $T\otimes_{\Pi} E_i \neq 0,$ then $T\lten_{\Pi}I_i\cong T\otimes _{\Pi}I_i\cong T I_i$ is  also a tilting $\Pi$-module of projective
dimension at most one.

\end{proposition}
\begin{proof}
 Applying $T\otimes_{\Pi}-$
to the exact sequence  \[0\to I_i\xrightarrow{f}\Pi\to E_i\to 0\] yields an exact sequence
 \[\Tor^{\Pi}_1(T,E_i)\to T\otimes _{\Pi}I_i\xrightarrow{T\otimes _{\Pi} f} T\otimes _{\Pi}\Pi\to T\otimes _{\Pi}E_i\to 0.\] 
It is clear that $\im (T\otimes _{\Pi} f)=TI_i$.

 $(a)$~If $T\otimes_{\Pi} E_i = 0$, then $T\otimes _{\Pi} f$ is surjective and $T\cong T\otimes _{\Pi}\Pi\cong \operatorname{Im} (T\otimes _{\Pi} f)=TI_i$. 

$(b)$~If $T\otimes_{\Pi} E_i \neq 0$,  we have  $\Tor^{\Pi}_1(T,E_i) = 0$ by Lemma \ref{lem of T times E}.
In particular, the morphism $T\otimes _{\Pi} f$ is injective and consequently $ T\otimes _{\Pi}I_i\cong TI_i$.
Since $\opname{proj.dim} T\leq 1$, then $\Tor^{\Pi}_1(T, I_i)=\Tor^{\Pi}_2(T, E_i)=0$ and $\Tor^{\Pi}_2(T, I_i)=0$. Hence $T\lten_{\Pi}I_i\cong T\otimes _{\Pi}I_i\cong TI_i.$   By Lemma \ref{key lem} , we know that $T\lten_{\Pi}I_i=T I_i$ is a tilting $\Pi$-module and $\End_{\Pi}(T I_i)\cong 
\End_{\Pi}(T)$.
 
Since $ T$ and $TI_i$ are  tilting modules, we deduce that $T$ and $TI_i$ are  locally free by Proposition  \ref{tilting are free}. Consider  the short exact sequence
\[0\to TI_i\to T\to T/{TI_i}\to 0,\]  
which implies that $T/{TI_i}$ is also locally free and  $\opname{proj.dim} (T/TI_i)\leq 2$ by Lemma \ref{nondynkin equivalent}, one can show that  $\opname{proj.dim}TI_i\leq 1$.  This finishes the proof.
\end{proof}

\begin{definition}Let $C$ be a symmetrizable Cartan matrix with a symmetrizer $D$ which has no component of Dynkin type and $\Pi=\Pi(C, D)$ the associated preprojective algebra.
An ideal $I$ of $\Pi$ is said to be {\it cofinite tilting ({\it resp.} partial tilting) } if $\Pi/I$ has finite length and $I$ is
  tilting ({\it resp.} partial tilting) as left $\Pi$-module and as right $\Pi-$module. 
  \end{definition}
Denote  by $\langle I_1,I_2,\cdots,I_n\rangle$ the semigroup generated by $I_1,I_2,\cdots,I_n$.
  The following result is a direct consequence of  Proposition ~\ref{Ii is tilting} and Proposition ~\ref{chengji is tilting}.
 
 \begin{proposition}
 Each $T\in \langle I_1,I_2,\cdots,I_n\rangle$ is a  cofinite tilting ideal and satisfies $\End_{\Pi}(T) \cong \Pi$.
\end{proposition}
 On the other hand, we have
  
\begin{lemma}\label{property of cofinte ideal}
Let $I$ be a cofinite tilting ideal of $\Pi$, then  $\opname{proj.dim} I\leq 1$ and $\Hom_{\Pi}(I,\Pi)\cong\Hom_{\Pi}(\Pi,\Pi)\cong\Pi.$
\end{lemma}
\begin{proof}
Consider the short exact sequence $0\to I\to \Pi\to \Pi/{I}\to 0.$
Note that  $I$ is an tilting module, which implies that $I$ is locally free by Proposition  \ref{tilting are free}.  Consequently,   $\Pi/I$ is locally free and $\opname{proj.dim} (\Pi/I)\leq 2$. Then it is easy to see that $\opname{proj.dim}  I\leq 1$.

Since  $\Pi/{I}$  is a locally free $\Pi$-module with finite length,  we clearly have $\Pi/{I}\in \rep_{l.f.}(\Pi)$. Therefore, $\Ext^{k}_{\Pi}( \Pi/{I},\Pi)$$ \cong
\mathbb{D}\Ext^{2-k}_{\Pi}(\Pi,\Pi/I)=0$
 for $k = 0, 1$ by Lemma \ref{2-cy}.  Applying
the functor $\Hom_{\Pi}(-,\Pi)$ to the above exact sequence, one obtains that $\Hom_{\Pi}(I, \Pi)\cong \Hom_{\Pi}(\Pi,\Pi)\cong \Pi$.
\end{proof}

\begin{proposition}
Each  cofinite partial tilting left or right ideal of $\Pi$ is a  cofinite tilting ideal and any   cofinite tilting ideal of $\Pi$ belongs to $ \langle I_1,I_2,\cdots,I_n\rangle$.
\end{proposition}
\begin{proof}
Let $ T$ be a  cofinite partial tilting right ideal of $\Pi$. By Lemma \ref{tilting are free}, $T$ is locally free.   If $T \neq \Pi$,  consider the locally free $\Pi$-module $\Pi/T$, which  has a generalized simple submodule, say $E_i$. Recall that we have $\Hom_{\Pi}(E_i,\Pi)=0=\Ext^1_{\Pi}(E_i,\Pi)$. By applying the functor $\Hom_{\Pi}(E_i,-)$ to the short exact sequence
\[0\to T\to \Pi\to \Pi/T\to 0,
\]
one obtains $\Hom_{\Pi}(E_i, T)=0$ and $\Ext^1_{\Pi}(E_i,T)\cong \Hom_{\Pi}(E_i, \Pi/T)\neq 0$. Consequently, $\Tor^{\Pi}_{1}(T, \mathbb{D}E_i)\cong D\Ext^1_{\Pi}(E_i, T)\neq 0$ and $T\otimes_{\Pi}\mathbb{D}E_i=0$ by Lemma~\ref{lem of T times E}.

Now put $U = \RHom_{\Pi}(I_i, T)$.  By Lemma \ref{key lem}, we deduce that $U\cong T\lten_{\Pi}\RHom_{\Pi}(I_i, T) $ is a
partial tilting complex of $\Pi$. Note that we have $\opname{proj.dim}  I_i~\leq 1$, it is easy to see that  \[\Ext^1_{\Pi}(I_i, T) \cong \Ext^2_{\Pi}(E_i, T)\cong \mathbb{D}\Hom_{\Pi}(T, E_i)\cong T\otimes_{\Pi}\mathbb{D}E_i=0.\]
In particular, $U \cong \Hom_{\Pi}(I_i, T)$ is a partial tilting $\Pi$-module. Moreover, by $U\cong\Hom_{\Pi}(I_i,T)\subset \Hom_{\Pi}(I_i,\Pi)\cong \Pi$, we may identify $U$ as a partial tilting right ideal of $\Pi$.

Now applying $\Hom_{\Pi}(-,T)$ to the  exact sequence of $\Pi$-bimodules
\[0\to I_i\to\Pi\to E_i\to 0\]
yields the following exact sequence of right $\Pi$-modules
\[0=\Hom_{\Pi}(E_i,T)\to \Hom_{\Pi}(\Pi,T)\to \Hom_{\Pi}(I_i,T)\to \Ext^1_{\Pi}(E_i,T)\to \Ext^1_{\Pi}(E_i,\Pi)=0.\]
In particular, we have \[0\to T\to U\to  \Ext^1_{\Pi}(E_i,T)\to 0. \]
Hence $U/T\cong  \Ext^1_{\Pi}(E_i,T)\cong \Hom_{\Pi}(E_i,\Pi/T)\neq 0$. For $j\neq i$, it is not hard to see that $\Hom_{\Pi}(e_j\Pi, U/T)=0$.
 Note that $U/T$ is also locally free,  we conclude that  $U/T$ is a direct sum of  copies of $E_i$, say $U/T\cong E_i^l$.  We can rewrite the short exact sequence as \[0\to T\to U\to E_i^{l}\to 0\]
and obtain an exact sequence \[0\to TI_i\to UI_i\to E_i^lI_i\to 0.\]
Since $E_iI_i=E_i\Pi(1-e_i)\Pi=E_i(1-e_i)\Pi=0$ and  $TI_i=T$ because of  $T\otimes_{\Pi}\mathbb{D}E_i=0$, we have $T=UI_i.$
Thus $T\in \langle I_1,I_2,\cdots,I_n \rangle$ by induction on the rank length of locally free module $\Pi/T$.
\end{proof}

\begin{proposition}
Let   $U,T\in \langle I_1,I_2,\cdots,I_n\rangle$ be two cofinite tilting ideals such that $U$ is isomorphic to $T$ as left modules or as right modules, then $T=U.$
\end{proposition}
\begin{proof} Assume that $U$ is isomorphic to $T$ as right $\Pi$-modules.
 Denote the isomorphisms by $f:U\to T$ and $g=f^{-1}:T\to U$. By Lemma \ref{property of cofinte ideal}, $f,g,fg,gf$ can be extended to morphisms from $\Pi$ to $\Pi$, still denoted by $f,g,fg,gf.$ Since $(fg)|_{T}=id_T,(gf)|_{U}=id_U$, then $fg=id_{\Pi}=gf$. Thus there exists an invertible element $x\in \Pi$ such that $f$ is the left multiplication with $x.$  Consequently, we have $T= f (U ) = xU=U.$ 
\end{proof}

In sum, we have proved the main result of this section.
\begin{theorem}
 Let $C$ be a symmetrizable Cartan matrix with a symmetrizer $D$ which has no component of Dynkin type and $\Pi=\Pi(C, D)$ the associated preprojective algebra.
 There is a bijection between  the set of  all cofinite tilting $\Pi$-ideals and the semigroup $\langle I_1,I_2,\cdots,I_n\rangle$. \hfill{$\Box$}
\end{theorem}

\section{Weyl Group and ideal semigroup}~\label{S:Weyl-group-and-ideal-semigroup}

As before, let $C = (c_{ij})$ be a symmetrizable Cartan matrix.
Denote by  $W=W(C)$  the Coxeter group generated by  $\{s_i~|~1\leq i\leq n\}$ with the following relations:

$(1)$ $s_i^2=id$;

$(2)$ $s_is_j = s_js_i$, if $c_{ij}c_{ji}=0$;

$(3)$ $s_is_js_i = s_js_is_j$, if $c_{ij}c_{ji}=1$;

$(4)$ $s_is_js_is_j = s_js_is_js_i$, if $c_{ij}c_{ji}=2$;

$(5)$ $s_is_js_is_js_is_j = s_js_is_js_is_js_i$, if $c_{ij}c_{ji}=3$.

By~\cite[Thm 16.17]{Ca},  $W$ is just the Weyl group of the Kac-Moody Lie algebra associated to the symmetrizable  Cartan matrix $C$.

For a word  $w=s_{i_1}s_{i_2}\cdots s_{i_l}\in W$, it is obvious that  the following operations on $w$ keep the word $w$ unchanged.

$(i)$ remove $s_is_i$;

$(ii)$ replace $s_is_j $ by $s_js_i$, if $c_{ij}c_{ji}=0$;

$(iii)$  replace $s_is_js_i $ by $ s_js_is_j$, if $c_{ij}c_{ji}=1$;

$(iv)$  replace $s_is_js_is_j $ by $ s_js_is_js_i$, if $c_{ij}c_{ji}=2$;

$(v)$  replace $s_is_js_is_js_is_j $ by $ s_js_is_js_is_js_i$, if $c_{ij}c_{ji}=3$.

Here we call the operation $(i)$ {\it nil-move} and the operations $(ii),(iii),(iv),(v)$ {\it braid-moves}.

The following result for Coxeter group is well-known ({\it cf.} Theorem 3.3.1 of ~\cite{BB}).
\begin{lemma}~\label{Word Property}
Let $W$ be a Coxeter group and $w\in W.$

$(a)$ Any expression $s_{i_1}s_{i_2}\cdots s_{i_l}$ for $w$ can be transformed into a reduced expression for $w$ by a sequence of nil-moves and braid-moves.

$(b)$ Every two reduced expressions for $w$ can be connected via a sequence of braid-moves.
\end{lemma}

Let $V^*$  be a vector space with basis $\alpha_1^*, \alpha_2^*,\cdots, \alpha_n^*$, define the linear transformation $\sigma_i^*\in GL{(V^*)}$ by $\sigma_i^*(p)=p-p_i(\sum\limits_{j=1}^nc_{ji}\alpha_j^*)$ for  $p=\sum\limits_{j=1}^np_j\alpha_j^*$. Obviously, we have
\[\sigma_i^*(\alpha_l^*)=\begin{cases}\alpha_i^*-\sum\limits_{j=1}^n c_{ji}\alpha_j^* & if\ l=i;  \\ \alpha_l^* & if\ l\neq i.\end{cases}\]

The following lemma gives a geometric realization of $W$({\it cf.} Theorem 4.2.7 of ~\cite{BB}).
\begin{lemma}\label{BB}
The mapping $s_i\mapsto \sigma_i^*$ for $i=1,2,\cdots,n$, extends uniquely to an injective  homomorphism $\varphi:W\to GL{(V^*)}$.
\end{lemma}

In the following, we are going to discuss the relationship between the Weyl group and the semigroup $\langle I_1,I_2,\cdots,I_n\rangle$. First, let us deal some special cases. Since the preprojective algebras $\Pi(C,D)$ does not depend on the orientation $\Omega$ of $C$,  we omit the orientation in the following.
\begin{proposition}\label{prop of preprojective of type A2}
Let $C=\left(\begin{array}{cc}2&-1\\-1&2 \end{array}\right)$ and $D=\opname{diag}(d,d)$ with positive integer $d$. The preprojective algebra $\Pi=\Pi(C,D)$ is given by the quiver
\[\xymatrix{
\ar@(ul,dl)_{\varepsilon_1}1\ar[rr]<1mm>^{a_{21}}&&\ar@(ur,dr)^{\varepsilon_2}2\ar[ll]^{a_{12}}
}
\]
with relations $\varepsilon_1^{d}=0, \varepsilon_2^{d}=0, \varepsilon_1a_{12}=a_{12}\varepsilon_2,\varepsilon_2a_{21}=a_{21}\varepsilon_1, a_{12}a_{21}=a_{21}a_{12}=0 $,
then $I_1I_2I_1=0=I_2I_1I_2$.
\end{proposition}
\begin{proof} Note that $I_1I_2I_1=\Pi(1-e_1)\Pi\Pi(1-e_2)\Pi\Pi(1-e_1)\Pi=\Pi e_2\Pi e_1\Pi e_2\Pi$. For any non-negative integer $l$, by $\varepsilon_2a_{21}=a_{21}\varepsilon_1$, we have $a_{21}\varepsilon_1^la_{12}=\varepsilon_2^la_{21}a_{12}=0$. In particular, $ e_2\Pi e_1\Pi e_2=0$ and $I_1I_2I_1=0$. Similarly, one can prove $I_2I_1I_2=0$.
\end{proof}

\begin{proposition}\label{prop of preprojective of type B2}
Let $C=\left(\begin{array}{cc}2&-1\\-2&2 \end{array}\right)$, $D=\opname{diag}(2d,d)$ with  positive integer $d$. The preprojective algebra $\Pi=\Pi(C,D)$ is given by the quiver
\[\xymatrix{
\ar@(ul,dl)_{\varepsilon_1}1\ar[rr]<1mm>^{a_{21}}&&\ar@(ur,dr)^{\varepsilon_2}2\ar[ll]^{a_{12}}
}
\]
with relations $\varepsilon_1^{2d}=0, \varepsilon_2^{d}=0, \varepsilon_1^2a_{12}=a_{12}\varepsilon_2,\varepsilon_2a_{21}=a_{21}\varepsilon_1^2, a_{12}a_{21}\varepsilon_1+\varepsilon_1a_{12}a_{21}=0, a_{21}a_{12}=0$,
then $I_1I_2I_1I_2=0=I_2I_1I_2I_1$.
\end{proposition}
\begin{proof}
 For any non-negative integers $l,k$, using the relations, we have $$a_{12}\varepsilon_2^ka_{21}\varepsilon_1^la_{12}=a_{12}a_{21}\varepsilon_1^{2k}\varepsilon_1^la_{12}=(-1)^{2k+l}\varepsilon_1^{2k+l}a_{12}a_{21}a_{12}=0,$$ and
$$a_{21}\varepsilon_1^la_{12}\varepsilon_2^ka_{21}=a_{21}\varepsilon_1^la_{12}a_{21}\varepsilon_1^{2k}=(-1)^la_{21}a_{12}a_{21}\varepsilon_1^{2k+l}=0.$$ Consequently, $ e_2\Pi e_1\Pi e_2\Pi e_1=0$ and $ e_1\Pi e_2\Pi e_1\Pi e_2=0$, which imply the desired equalities $I_1I_2I_1I_2=\Pi e_2\Pi e_1\Pi e_2\Pi e_1\Pi=0$ and $I_2I_1I_2I_1=\Pi e_1\Pi e_2\Pi e_1\Pi e_2\Pi=0$.
\end{proof}
\begin{proposition}\label{prop of preprojective of type F4}
Let $C=\left(\begin{array}{cc}2&-1\\-3&2 \end{array}\right)$, $D=\opname{diag}(3d,d)$ with positive integer $d$. The preprojective algebra $\Pi=\Pi(C,D)$ is given by the quiver
\[\xymatrix{
\ar@(ul,dl)_{\varepsilon_1}1\ar[rr]<1mm>^{a_{21}}&&\ar@(ur,dr)^{\varepsilon_2}2\ar[ll]^{a_{12}}
}
\]
with relations $\varepsilon_1^{3d}=0, \varepsilon_2^{d}=0,  \varepsilon_1^3a_{12}=a_{12}\varepsilon_2,\varepsilon_2a_{21}=a_{21}\varepsilon_1^3, a_{12}a_{21}\varepsilon_1^2+\varepsilon_1a_{12}a_{21}\varepsilon_1+\varepsilon_1^2a_{12}a_{21}=0, a_{21}a_{12}=0 $,
then $I_1I_2I_1I_2I_1I_2=0=I_2I_1I_2I_1I_2I_1$.
\end{proposition}
\begin{proof}
The proof involves similar but more complicate calculation as Proposition~\ref{prop of preprojective of type B2} and we left the calculation as an exercise to the reader.
\end{proof}

Now we are in a position to consider the general cases.
\begin{proposition}\label{relation of I_i}

Let $\Pi=\Pi(C,D)$ be a preprojective algebra, then

$(a)$ $I^2_i = I_i$;

$(b)$ $I_iI_j = I_jI_i$, if $c_{ij}c_{ji}=0$;

$(c)$ $I_iI_jI_i = I_jI_iI_j$, if $c_{ij}c_{ji}=1$;

$(d)$ $I_iI_jI_iI_j = I_jI_iI_jI_i$, if $c_{ij}c_{ji}=2$;

$(e)$ $I_iI_jI_iI_jI_iI_j = I_jI_iI_jI_iI_jI_i$, if $c_{ij}c_{ji}=3$.

\end{proposition}
\begin{proof}
$(a)$ is obvious. 

Let $I_{i,j}=\Pi(1-e_i-e_j)\Pi$, then any product of ideals $I_i$ and $I_j$ contains $I_{i,j}$.

If $c_{ij}c_{ji}=0$, then
$\Pi/I_{i,j}$ is semisimple, thus $I_iI_j =I_{i,j}= I_jI_i.$

If $c_{ij}c_{ji}=1$, then $\Pi/I_{i,j}$ is isomorphic to the preprojective algebra  in Proposition  \ref{prop of preprojective of type A2}. Let $I$ be an ideal of $\Pi$, denote the image of $I$ in $\Pi/I_{i,j}$ by $\overline{I}$. Hence, by Proposition  \ref{prop of preprojective of type A2}, we have $\overline{I_iI_jI_i} =\overline{I_i}\ \overline{I_j}\ \overline{I_i}=0 = \overline{I_j}\ \overline{I_i}\ \overline{I_j}= \overline{I_jI_iI_j}$, thus  $I_iI_jI_i \subset I_{i,j}$ and $I_jI_iI_j\subset I_{i,j}$.  Since any product of ideals $I_i$ and $I_j$ contains $I_{i,j}$, then we have $I_iI_jI_i =I_{i,j}= I_jI_iI_j$.

If $c_{ij}c_{ji}=2$ or $3$,   then $\Pi/I_{i,j}$ is isomorphic to the preprojective algebra  in Proposition \ref{prop of preprojective of type B2} or the preprojective algebra  in Proposition \ref{prop of preprojective of type F4}. Similar to  the case $c_{ij}c_{ji}=1$, one can prove the statements $(d)$ and $(e)$.
\end{proof}

\begin{theorem}\label{bijection w and I nondynkin}
Let $\Pi=\Pi(C,D)$ be a preprojective algebra.
 There exists a bijection $\psi:W\to  \langle I_1,I_2,\cdots,I_n\rangle$ given by $\psi(w)=I_w = I_{i_1}I_{i_2}\cdots I_{i_k}$
for any reduced expression $w = s_{i_1}s_{i_2}\cdots s_{i_k}$.
\end{theorem}
\begin{proof}
We first show that the map $\psi$ is well-defined.
 For any two reduced expression $w_1,w_2$ of an element  in $ W$, by Lemma \ref{Word Property}, $w_1, w_2$ can be connected by a sequence of braid-moves. Then by Proposition \ref{relation of I_i},  we have $I_{w_1}=I_{w_2}$.
 
 For any ideal $I\in  \langle I_1,I_2,\cdots,I_n\rangle$, take $I = I_{i_1}I_{i_2}\cdots I_{i_k}$ with $k$ minimal,  let $w=s_{i_1}s_{i_2}\cdots s_{i_k}$.
By Lemma \ref{Word Property}, $w$ can be transformed into a reduced expression by a sequence of nil-moves and braid-moves. By Proposition \ref{relation of I_i}, the nil-moves can not appear since $k$ is minimal. Hence $w$ is a reduced expression and we have $\psi(w)=I_w=I$. In particular, the map is surjective.

It remains to show that the map is injective. We separate the proof into two cases.

Let us assume first that $C$ has no component of Dynkin type.
Denote by $\varepsilon =\mathcal{K}^b(\operatorname{proj} \Pi)$, {\it i.e.} the homotopy category of bounded complexes of projective $\Pi$-modules. For any $i\in Q_0$, since $I_i$ is a $\Pi$-tilting module with $\End~ I_i\cong\Pi$, we have an autoequivalence $-\lten_{\Pi}I_i$
of $\varepsilon$  which induces an automorphism $[-\lten_{\Pi}I_i]$
of the Grothendieck group $K_0(\varepsilon)$. Note that $\{[e_j\Pi]~|~j\in Q_0\}$ is a $\mathbb{Z}$-basis of $K_0(\varepsilon)$.
By the projective resolution of $e_iI_i$
 $$0\to e_i\Pi\to \bigoplus_{j\in\overline{\Omega}(i,-)}(e_j\Pi)^{|c_{ji}|}\to e_iI_i\to 0,$$
it is easy to get that 
$$[e_l\Pi\lten_{\Pi}I_i]=\begin{cases}[e_iI_i]=[\bigoplus\limits_{j\in \overline{\Omega}(i,-)}(e_j\Pi)^{|c_{ji}|}]-[e_i\Pi]=[e_i\Pi]-\sum\limits_{j=1}^nc_{ji}[e_j\Pi]& if\ l=i;  \\ [e_lI_i] =[e_l\Pi] & if\ l\neq i. \end{cases}$$
In particular, by  identifying  $-\lten_{\Pi}I_i$ with  $\sigma_i^* $ for $V^*=K_0(\varepsilon)\otimes_{\mathbb{Z}} \mathbb{C}$, 
we obtain an action $s_i\mapsto [-\lten_{\Pi}I_i]$ of $W$ on $V^*$. 
Note that for any reduced expression  $w = s_{i_1}s_{i_2}\cdots s_{i_k}$, $I_w = I_{i_1}I_{i_2}\cdots I_{i_k}=I_{i_1}\lten_{\Pi}I_{i_2}\lten_{\Pi}\cdots \lten_{\Pi} I_{i_k}.$
Hence, the action of $[-\lten_{\Pi}I_w]$ on $V^*$ coincides with the action of $ [-\lten_{\Pi}I_{i_1}\lten_{\Pi}I_{i_2}\lten_{\Pi}\cdots \lten_{\Pi} I_{i_k}]$ on $V^*$.
Therefore, for any reduced word $w, w'$ such that  $I_w=I_{w'}$, we have $[-\lten_{\Pi}I_w]=[-\lten_{\Pi}I_{w'}]$. By
Lemma \ref{BB}, we deduce that $w=w'$.

Now suppose that   $C$ is  of Dynkin type and without  loss of generality, we may assume that $C$ is connected.
Let $\widetilde{Q}=\widetilde{Q}(\widetilde{C},\widetilde{\Omega})$ be an extended Dynkin quiver obtained from $Q=Q(C,\Omega)$ by adding a new vertex  $(i.e.\ \widetilde{Q_0}=\{0\}\cup Q_0)$ and the associated arrows, $\widetilde{\Pi}$ be the associated preprojective algebra and  $\widetilde{W}$ the  associated Weyl group.
Denote  $\widetilde{I_i}=\widetilde{\Pi}(1-e_i)\widetilde{\Pi}$ for each  $i\in \widetilde{Q_0}.$
For each $w\in  \widetilde{W}$,  let $w=s_{i_1}\cdots s_{i_k}$ be a reduced expression,
 denote $\widetilde{I_{ w}}=\widetilde{I}_{i_1}\cdots \widetilde{I}_{i_k}$.
Then we have $\Pi=\widetilde{\Pi}/\langle e_0\rangle, I_i=\widetilde{I_i}/\langle e_0\rangle$ for $i\neq 0$. Hence for each reduced word $w=s_{i_1}\cdots s_{i_k}\in  {W}$, we have $$I_{w}=({\widetilde{I}}_{i_1}/\langle e_0\rangle)\cdots ({\widetilde{I}}_{i_k}\langle e_0\rangle)=({\widetilde{I}}_{i_1}\cdots {\widetilde{I}}_{i_k})/\langle e_0\rangle=\widetilde{I}_{w}/\langle e_0\rangle.$$
Suppose $I_w=I_{w'}$ for two reduced words $w,w'\in W$, then $\widetilde{I}_{w}/\langle e_0\rangle=\widetilde{I}_{w'}/\langle e_0\rangle$.
Since $s_0$ does not appear in reduced expression of $w$ and $w'$, we have $e_0\in \widetilde{I}_{w}$ and $e_0\in \widetilde{I}_{w'}$.
Consequently, $\widetilde{I}_{w}=\widetilde{I}_{w'}$ and $w=w'$ by the case of non-Dynkin type.
\end{proof}

\section{Preprojective algebras of Dynkin type}~\label{S:Dynkin-type}

In this section, after recall the basic definition and properties of support $\tau$-tilting modules, we generalize the classification results of \cite{M} for classical preprojective algebras of type $A,D,E$ to the preprojective algebras of Dynkin type in the sense of ~\cite{GLS14}. We mainly follow~\cite{M}.
\subsection{Recollection on support $\tau$-tilting modules}
In this subsection, we assume $\Lambda$ to be a finite-dimensional algebra over $K$. Let $\tau$ be the Auslander-Reiten translation of $\Lambda$-modules. We follow~\cite{AIR}.

\begin{definition}Let $M$ be a $\Lambda$-module and $P$ a  projective $\Lambda$-module.

$(1)$  $M $ is  a {\it $\tau$-rigid} $\Lambda$-module if $\Hom_{\Lambda}(M,\tau M) = 0.$

$(2)$  $M$ is a {\it $\tau$-tilting} $\Lambda$-module if it is $\tau$-rigid and $|M| = |\Lambda|.$

$(3)$  $M$ is a {\it support $\tau$-tilting} $\Lambda$-module if there exists an idempotent $e \in \Lambda$ such
 $M$ is a $\tau$-tilting $(\Lambda/\langle e\rangle)$-module.
 
$(4)$ The pair $(M,P)$ is a {\it $\tau$-rigid pair} if $M$ is $\tau$-rigid and
$\Hom_{\Lambda}(P,M) = 0.$

$(5)$ A $\tau$-rigid pair $(M,P)$ is a {\it support $\tau$-tilting} (respectively, {\it almost complete
support $\tau$-tilting}) pair if $|M| + |P| = |\Lambda|$ (respectively, $|M| + |P| = |\Lambda| -1$).
\end{definition}

The pair $(M,P)$ is {\it basic} if $M$ and $P$ are basic. By \cite[Prop 2.3]{AIR},  if $(M,P)$ is a basic support $\tau$-tilting pair, then $P$ is determined by $M$ uniquely, and $(M,P)$ is a 
support $\tau$-tilting pair if and only if 
$M$ is a $\tau$-tilting  $(\Lambda/\langle e\rangle)$-module, where $e$ is an idempotent of $\Lambda$ such that $\add P =\add e\Lambda $. Thus we can identify basic support $\tau$-tilting modules with basic support $\tau$-tilting pairs.

For a $\Lambda$-module $M$, denote by $\opname{Fac} M$ the category formed by the quotients of finite direct sum of $M$.
Let $\opname{s\tau-tilt} \Lambda$ be the set of isomorphism classes of basic support $\tau$-tilting $\Lambda$-modules. For $T, T'\in \opname{s\tau-tilt} \Lambda$, we write $T\leq  T'$ if $\opname{Fac} T\subseteq \opname{Fac} T'$. This defines a partial order $\leq$ on $\opname{s\tau-tilt} \Lambda$.

The following result has been obtained in~\cite{AIR}.
\begin{theorem}
Any basic almost support $\tau$-tilting pair $(U,Q)$ is  a direct summand of exactly two basic support $\tau$-tilting pairs $(T,P),(T',P')$. Moreover,  we have either $T<T'$ or $T'<T$.
\end{theorem}
In the situation of the above theorem, we say $(T', P')$ is a {\it left mutation} of $(T, P)$ if $T'<T$ and write $\mu_X^-(T)=T'$, where $X$ is the unique indecomposable direct summand of $(T, P)$ which is different from $(T', P')$. In this case, we also call $(T, P)$ is a {\it right mutation} of $(T', P')$ at $X$. Let $T=U\oplus X$ be a support $\tau$-tilting $\Lambda$-module with indecomposable direct summand $X$. It is also known that $T$ has a left mutation at $X$ if and only if $X\not\in \opname{Fac} U$.
 The following result  also gives us a method to calculate left mutations of support $\tau$-tilting modules.
\begin{lemma}\cite[Theorem 2.30]{AIR}\label{mutaion}
 Let $X$ be an indecomposable $\Lambda$-module and $T = X \oplus U$ a basic $\tau$-tilting $\Lambda$-module. Assume that $T$ has a left mutation  $\mu_X^{-}(T)$. Let $$X \xrightarrow{f} U '\xrightarrow{g} Y \to 0$$ be an exact sequence, where $f $ is a minimal left $\operatorname{add} U$-approximation. 
 Then one of the following holds:
 
(1) $Y =0.$  Then $U=\mu_{X}^-(T)$ is a basic support $\tau$-tilting module.
 
(2) $Y\neq 0$ and $ Y \cong Y'^m $ for some integer $m>0$, where $Y '$ is an indecomposable direct summand of $Y$. Then $\mu_X ^-(T ) = Y ' \oplus U$  is a basic $\tau$-tilting module. 
\end{lemma}

Using the left mutation, one may define the  support $\tau$-tilting quiver of $\Lambda$.
\begin{definition}
 The {\it support $\tau$-tilting quiver} $\mathcal{H}(\opname{s\tau-tilt}\Lambda)$ is defined as follows.

$(1)$ The set of vertices is $\opname{s\tau-tilt}\Lambda$;

$(2)$ Draw an arrow from $T$ to $T'$ if $T'$ is a left mutation of $T$.
\end{definition}
\begin{lemma}\cite[Corollary 2.38]{AIR}\label{connect}
 If $\mathcal{H}(\opname{s\tau-tilt}\Lambda)$  admits a finite connected component $\mathcal{C}$, then $\mathcal{H}(\opname{s\tau-tilt}\Lambda) = \mathcal{C}$.  
\end{lemma}

\subsection{Support $\tau$-tilting modules for preprojective algebras of Dynkin type}
In this subsection, we always assume that $C$ is a symmetrizable Cartan matrix of Dynkin type with a symmetrizer $D$ and $\Pi=\Pi(C,D)$ is the associated preprojective algebra. We follow~\cite{M}.
 By Lemma \ref{selfinjective}, we know that $\Pi$
is a selfinjective algebra and the Nakayama functor  $\nu=\mathbb{D}\Hom_{\Pi}(-,\Pi)$ is exact.

\begin{lemma}\label{indec}
Let $I$ be a two-sided ideal of $\Pi$. For a primitive idempotent $e_i$ for $i \in Q_0$,
$e_iI$ is either indecomposable or zero.
\end{lemma}
\begin{proof}
 Since $\Pi$ is selfinjective, $e_i\Pi$ has a unique simple socle. If the submodule $e_iI$ is non-zero, then it has a unique simple socle as $e_i\Pi$ and hence indecomposable.
\end{proof}

An easy consequence of the above result is that for $i\neq j$,  $e_iI$ and $e_jI$ are non-isomorphic provided that they are not both zero.
\begin{proposition}~\label{Eiistaurigid}
For each $i\in Q_0$, the generalized simple module  $E_i$ and the two-sided ideal $I_i$ are  $\tau$-rigid modules.
\end{proposition}
\begin{proof}
Recall that $I_i=\bigoplus\limits_{j=1}^n e_jI_i=e_iI_i\oplus( \bigoplus\limits_{j\neq i} e_j\Pi)$. If $e_iI_i=e_i\Pi(1-e_i)\Pi=0$, we deduce that $i$ is an  isolated vertex. In this case, we have $I_i=\bigoplus_{j\neq i} e_j\Pi$ and $E_i=e_i\Pi$, then the result is obvious. 

Now assume that $e_iI_i\neq 0.$
 By Lemma \ref{lem of EI resolution}, for any $i$, we have an exact sequence 
 \begin{eqnarray}0\to E_{\sigma(i)}\to e_i\Pi\xrightarrow{a} \bigoplus_{j\in \overline{\Omega}(i,-)}(e_j\Pi)^{|c_{ji}|}\xrightarrow{b} e_i\Pi\xrightarrow{c} E_i\to 0,\notag\end{eqnarray}
 where $\ker c=e_iI_i.$ 
Applying the exact functor  $\nu=\mathbb{D}\Hom_{\Pi}(-,\Pi)$ to the above exact sequence,  one obtains the following exact sequence
\begin{eqnarray}0\to \nu({E_{\sigma(i)}})\to \nu({e_i\Pi})\xrightarrow{\nu(a)} \bigoplus_{j\in \overline{\Omega}(i,-)}\nu({e_j\Pi})^{|c_{ji}|}\xrightarrow{\nu(b)}\nu({e_i\Pi})\xrightarrow{\nu(c)} \nu({E_i})\to 0,\notag\end{eqnarray}
By the definition of the functor $\tau$, we know that $\tau(E_i)=\ker \nu(b)$ and $\tau(e_iI_i)=\ker \nu(a)=\nu({E_{\sigma(i)}})$.
Recall that we have $\nu(E_{\sigma(i)})= E_i$ by Proposition \ref{ci=csigmai}.
Therefore,  $\tau(I_i)=\tau(e_iI_i\oplus( \bigoplus_{j\neq i} e_j\Pi))=\tau(e_iI_i)= E_i $. Consequently, 
  $\Hom_{\Pi}(I_i,\tau(I_i))= \Hom_{\Pi}(I_i,E_i) =0$  and $I_i$ is a $\tau$-rigid module.
  
  To show $\Hom_{\Pi}(E_i, \tau E_i)=0$, note that \[\Hom_{\Pi}(E_i,\bigoplus\limits_{j\in \overline{\Omega}(i,-)}\nu({e_j\Pi})^{|c_{ji}|})\cong \Hom_{\Pi}((\Pi~e_j)^{|c_{ji}|}, \mathbb{D}(E_i))=0.\]
We conclude that $\Hom_{\Pi}(E_i,\tau(E_i))=\Hom_{\Pi}(E_i,\ker\nu(b))=0$ and $E_i$ is $\tau$-rigid.
\end{proof}
We have the following easy consequence of~Proposition~\ref{Eiistaurigid}.
\begin{proposition}
For each $i\in Q_0$,  the two-sided ideal $I_i$ is a basic support $\tau$-tilting ideal.
\end{proposition}
\begin{proof}
If $I_i=0$, then $I_i$ is obviously a basic support $\tau$-tilting module. 

Now  suppose that  $I_i\neq 0$.
 If  $e_iI_i=0,$  then $i$ is an  isolated vertex and  $(I_i,e_i\Pi)=(\bigoplus\limits_{j\neq i} e_j\Pi, e_i\Pi)$ is a support $\tau$-tilting pair, hence $I_i$ is a support $\tau$-tilting module.
If $e_iI_i\neq 0$,  it is clear that $|I_i|=|\Pi|$. Consequently, $I_i$ is a $\tau$-tilting module by Propposition \ref{Eiistaurigid}.
\end{proof}

By~\cite[Prop 2.5]{AIR}, for each $\tau$-rigid module $X$,  we can take a minimal projective presentation $P_1\to
P_0 \to X \to 0$ such that $P_0$ and $P_1$ do not have a common direct summand. Then similar to  Lemma \ref{lem of T times E}, one can prove
\begin{lemma}\label{TotimesEi}
 Let $T$ be a support $\tau$-tilting $\Pi$-module. For any generalized simple $\Pi^{op}$-module $E_i$, at least
one of the statements $T \otimes_{\Pi}E_i = 0$ and $\Tor_{\Pi}^1(T, E_i) = 0$ holds.
\end{lemma}

Similar to Proposition \ref{chengji is tilting},  applying the functor $T\otimes_{\Pi}-$
to the exact sequence \[0\to I_i\xrightarrow{f}\Pi\to E_i\to 0,\] we  obtain an exact sequence
 \[\Tor^{\Pi}_1(T,E_i)\to T\otimes _{\Pi}I_i\xrightarrow{T\otimes _{\Pi}f} T\otimes _{\Pi}\Pi\to T\otimes _{\Pi}E_i\to 0.\]
Note that $\im(T\otimes _{\Pi}f)=TI_i$, we have $T \otimes_{\Pi}E_i = 0$ if and only if $TI_i=T$.
In particular, 
\begin{lemma}\label{TI_i}
Let $T$ be a support $\tau$-tilting $\Pi$-module. If $TI_i \neq T$, then $T \otimes_{\Pi}I_i \cong
TI_i$. 
\end{lemma}

For  an ideal  $T$ of $\Pi$, if  $e_iT=0$, then $I_iT=\Pi(1-e_i)\Pi T=T$.  Thus if $I_iT\neq T$, then $e_iT\neq 0$.
The same proof of~\cite[Lemma 2.8]{M} yields the following.
\begin{lemma}\label{has left mutation}
Let $T$ be a support $\tau$-tilting ideal of $\Pi$. If $I_iT \neq T$, then we have $e_iT \notin
\opname{Fac}((1-e_i)T)$.  In particular, $T$ has a left mutation
$\mu^-_{e_iT}(T).$
\end{lemma}
Recall that we have the following two exact sequences 
\begin{eqnarray}0\ra E_{\sigma(i)}\to e_i\Pi\to \bigoplus_{j\in \overline{\Omega}(i,-)}(e_j\Pi)^{|c_{ji}|},\label{Esigmai}\end{eqnarray} and
\begin{eqnarray} e_i\Pi\xrightarrow{a} \bigoplus_{j\in \overline{\Omega}(i,-)}(e_j\Pi)^{|c_{ji}|}\to e_iI_i\to 0,\label{eiIi}\end{eqnarray}
which are obtained from the exact sequence (\ref{resolution of E_I dynkin case}).
The following is an analogue of Lemma~$2.9$ in~\cite{M}.
\begin{lemma}\label{left approximation}
Let $T$ be a support $\tau$-tilting ideal of $\Pi$.  The map $a\otimes_{\Pi}T$ 
\[ e_i\Pi\otimes_{\Pi}T\xrightarrow{a\otimes_{\Pi}T} \bigoplus_{j\in\overline{\Omega}(i,-)}(e_j\Pi)^{|c_{ji}|}\otimes_{\Pi}T\]
is a left $\add((1-e_i)T)$-approximation.
\end{lemma}
\begin{proof} 
Since the socle of $e_j\Pi$ is the simple module $S_{\sigma(j)}$,  we have  $\Hom_{\Pi}(E_{\sigma(i)},(1-e_i)\Pi)=0$.  By applying the functor $\Hom_{\Pi}(-,(1-e_i)\Pi)$ to the exact sequence (\ref{Esigmai}), we obtain a surjective map $\Hom_{\Pi}( \bigoplus\limits_{j\in \overline{\Omega}(i,-)}(e_j\Pi)^{|c_{ji}|},(1-e_i)\Pi)\to \Hom_{\Pi}(e_i\Pi,(1-e_i)\Pi) $. Then one can apply the same argument of Lemma~$2.9$ in~\cite{M} to obtain the result.
\end{proof}

\begin{proposition}
Assume that $T \in \langle I_1, \cdots , I_n\rangle $ is a basic support $\tau$-tilting
ideal of $\Pi$. Then $I_iT$ is a basic support $\tau$-tilting $\Pi$-module.
\end{proposition}
\begin{proof} 
There is nothing to prove if $I_iT=T$ and we assume that $I_iT\neq T$ in the following.
By Lemma \ref{has left mutation}, $T$ has a left mutation
$\mu^-_{e_iT}(T).$ Now let $e$ be an idempotent of $\Pi$ such that $T$ is a $\tau$-tilting $(\Pi/{\langle e\rangle})$ module.
Applying $-\otimes_{\Pi}T $ to the exact sequence (\ref{eiIi}), 
 we  get an exact sequence$$e_i\Pi\otimes_{\Pi}T\xrightarrow{a\otimes_{\Pi}T} \bigoplus_{j\in \overline{\Omega}(i,-)}(e_j\Pi)^{|c_{ji}|}\otimes_{\Pi}T\to e_iI_i\otimes_{\Pi}T\to 0.$$
and  $a\otimes_{\Pi}T$ is a left $\add ((1-e_i)T)$-approximation by Lemma \ref{left approximation}.
On the other hand, 
by  the left module version of Lemma \ref{TI_i}, we have $e_iI_i\otimes_{\Pi}T=e_iI_i T$.  If $e_iI_i T=0$, then it is clear that $a\otimes_{\Pi}T$ is a minimal left $\add ((1-e_i)T)$- approximation.
Now assume that $e_iI_i T\neq 0$. Since  $e_jT$ and $e_iI_iT$ have different socles  for $j\neq i$, then $e_jT$ and $e_iI_iT$ are non isomorphic  for $j\neq i$.
Hence the map $a\otimes_{\Pi}T$ is also a minimal left $\add ((1-e_i)T)$-approximation.
Consequently,  $\mu_{e_iT}^-(T)= e_iI_iT\oplus (1-e_i)T=I_iT$ is a $\tau$-tilting $(\Pi/{\langle e\rangle})$ module by Lemma \ref{mutaion} and $I_iT$ is a basic support $\tau$-tilting $\Pi$-module.
\end{proof}
 In particular, we have proved the following result.
 \begin{theorem}
Each  $T\in \langle I_1,\cdots,I_n\rangle$ is a basic support $\tau$-tilting modules. \end{theorem}

Let $W=W(C)$ be  the Weyl group associated to the symmetrizable Cartan matrix $C$. By Theorem \ref{bijection w and I nondynkin},  for each reduced expression of $w\in W$, we have a support $\tau$-tilting module $I_w$.
Let $(I_w,P_w)$ be the corresponding support $\tau$-tilting pair, where $P_w$ is a projective $\Pi$-module, then we have
\begin{proposition}
Keep notations as above, then \[P_w=\bigoplus_{i\in Q_0, e_iI_w=0}e_{\sigma(i)}\Pi,\] where $\sigma:Q_0\ra Q_0$ is the Nakayama permutation of $\Pi$.
\end{proposition}
\begin{proof}
Suppose $e_iI_w=0$ for  $i\in Q_0$, we need to show $\Hom_{\Pi}(e_{\sigma(i)}\Pi, I_w)=I_we_{\sigma(i)}=0$.  Since $I_w$ is an ideal,  $I_we_{\sigma(i)}$ is a left $\Pi$-module.
If  $I_we_{\sigma(i)}\neq 0$, then  $\operatorname{soc} (I_we_{\sigma(i)})=\operatorname{soc} (\Pi e_{\sigma(i)})= S_i$. Consequently,
there is a nonzero morphism $\Pi e_i\twoheadrightarrow S_i\rightarrowtail I_we_{\sigma(i)}.$
However, $\Hom_{\Pi}(\Pi e_i, I_we_{\sigma(i)})= e_iI_we_{\sigma(i)}=0$,  a contradiction. Hence $\Hom_{\Pi}(e_{\sigma(i)}\Pi, I_w)=0$ and $e_{\sigma(i)}\Pi$ is a direct summand of $P_w$.

On the other hand, suppose we have  $\Hom_{\Pi}(e_{\sigma(i)}\Pi, I_w)=I_we_{\sigma(i)}= 0$, using similar argument, one can get $e_iI_w=0$. This completes the proof. 
\end{proof}

\begin{theorem}\label{mutation of tau tilting}
Keep notations as above, for each $i\in Q_0$,  the support $\tau$-tilting pairs $(I_w,P_w)$ and $(I_{s_iw},P_{s_iw})$ are  related by a left or right mutation.
\end{theorem}
\begin{proof}
We need to show that the support $\tau$-tilting pairs $(I_w,P_w)$ and $(I_{s_iw},P_{s_iw})$ have a common almost complete support $\tau$-tilting pair.
According to Theorem~\ref{bijection w and I nondynkin}, $I_{s_iw}$ and $I_w$ are not isomorphic.  Suppose $l(s_iw)>l(w)$, we have $I_{s_iw}=I_iI_w$ and
\[I_{s_iw}= e_iI_iI_w\oplus(1-e_i)I_iI_w=e_iI_iI_w\oplus(1-e_i)I_w,~~ P_{s_iw}=\bigoplus_{j\in Q_0,e_jI_iI_w=0}e_{\sigma(j)}\Pi.\]
Note that $I_w= e_iI_w\oplus(1-e_i)I_w $ and $ P_w=\bigoplus\limits_{j\in Q_0, e_jI_w=0}e_{\sigma(j)}\Pi.$  By $I_iI_w\not\cong I_w$, it is not hard to see that $e_iI_w\neq 0$.
On the other hand,  for $j\neq i$, we clearly have $e_jI_iI_w=e_jI_w$ and hence $e_jI_iI_w=0$ if and only if $ e_jI_w=0.$  Then we obtain 
\[(I_{s_iw},P_{s_iw})=\begin{cases}((1-e_i)I_w,P_w\oplus e_{\sigma(i)}\Pi)& \text{if $ e_iI_iI_w=0$;}\\(e_iI_iI_w\oplus (1-e_i)I_w,P_w) &\text{if $e_iI_iI_w\neq 0$.}\end{cases}\]
 In particular, $(I_w,P_w)$ and $(I_{s_iw},P_{s_iw})$ have a common almost complete support $\tau$-tilting pair $((1-e_i)I_w, P_w)$ as a direct summand in both cases. 
 
If $l(s_iw)<l(w)$, one applies the same argument to $u:=s_iw$ which satisfies $l(s_iu)>l(u)$.
\end{proof}

\begin{theorem}~\label{t:support-tau-tilting}
Each basic support $\tau$-tilting $\Pi$-module is isomorphic to an object of $\langle I_1,I_2,\cdots, I_n\rangle.$ 
\end{theorem}
\begin{proof}
The same argument of~\cite[Thm 2.19]{M} applied. Here we provide a sketch of  the proof.

Let $\mathcal{C}$ be the connected component of $\mathcal{H}(\opname{s\tau-tilt}\Pi)$ containing $\Pi$.  By induction on the length of the element of the Weyl group $W$, it is easy to get  that $I_w$   belongs to $\mathcal{C}$ for any reduced word $w\in W$.  Then by Theorem \ref{mutation of tau tilting}, each neighbor of $I_w$ has the form $I_{s_iw}$ for some $i$,  we conclude  that  the set $\{I_w~|~w\in {W}\}$ forms the  vertices of $\mathcal{C}$ and $\mathcal{C}$ is a finite connected component of $\mathcal{H}(\opname{s\tau-tilt}\Pi)$ . Then by Lemma \ref{connect}, we deduce that $\mathcal{H}(\opname{s\tau-tilt}\Pi)=\mathcal{C}$.
\end{proof}
\begin{remark}
An immediately consequence of ~Theorem~\ref{t:support-tau-tilting} is that each indecomposable $\tau$-rigid module has the form $e_iI_w$ for some $i\in Q_0$ and some reduced word $w\in W$. On the other hand, for any $i\in Q_0$ and any reduced word $w\in W$, if $e_iI_w\neq 0$,  then  $e_iI_w$ is a nonzero indecomposable $\tau$-rigid module.
\end{remark}

\begin{example}
Let $\Pi=\Pi(C,D)$ be the preprojective algebra in Example \ref{eg1}. In this case, $\mathcal{H}(\opname{s\tau-tilt}\Pi)$ is as follows:
\[\xymatrix{
&E_2\oplus e_2\Pi\ar[rr]^{I_2}&&E_2\ar[dr]^{I_1}&\\
\Pi\ar[ur]^{I_1}\ar[dr]_{I_2}&&&&0\\
&E_1\oplus e_1\Pi\ar[rr]_{I_1}&&E_1\ar[ur]_{I_2}&
}\]
 and all the basic support $\tau$-tilting  pairs are 
 \[\{(\Pi,0), (E_2\oplus e_2\Pi,0),(E_1\oplus e_1\Pi,0), (E_1,e_2\Pi),(E_2,e_1\Pi),(0,\Pi)\},\]
 where  \[\xymatrix@=8pt{e_1\Pi={\small\begin{array}{c}1\\1\ \ 2\\2\end{array} },e_2\Pi={\small\begin{array}{c}2\\1\  \ 2\\1\end{array} },
E_1={\small\begin{array}{c}1\\1\end{array} }, E_2={\small\begin{array}{c}2\\2\end{array} .}
  }\]
  We also know that all the nonzero indecomposable $\tau$-rigid $\Pi$-modules  are  $\{e_1\Pi, e_2\Pi, E_1,E_2\}.$
\end{example}

\begin{example} 
Let $\Pi=\Pi(C,D)$ be the preprojective algebra in Example \ref{eg2}. 
  In this case, $\mathcal{H}(\opname{s\tau-tilt}\Pi)$ is as follows:
 \[\xymatrix{
&e_1I_1\oplus e_2\Pi\ar[rr]^{I_2}&&e_1I_1\oplus E_2\ar[rr]^{I_1}&&E_2\ar[dr]^{I_2}&\\
\Pi\ar[ur]^{I_1}\ar[dr]_{I_2}&&&&&&0\\
&e_2I_2\oplus e_1\Pi\ar[rr]_{I_1}&& e_2I_2\oplus E_1\ar[rr]_{I_2}&&E_1\ar[ur]_{I_1}&
}\]
 and all the basic support $\tau$-tilting pairs are 
 \[\{(\Pi,0), (e_1I_1\oplus e_2\Pi,0),(e_2I_2\oplus e_1\Pi,0),(E_1\oplus e_2I_2,0),(e_1I_1\oplus E_2,0), (E_1,e_2\Pi),(E_2,e_1\Pi),(0,\Pi)\},\]
 where \[ \xymatrix@=8pt{e_1\Pi={\small\begin{array}{c}1\\1\ \ 2\\2\ \ 1\\1\end{array} },e_2\Pi={\small\begin{array}{c}2\\1\\1\\2\end{array} },
 e_1I_1={\small\begin{array}{c}1\ \ 2\\2\ \ 1\\1\end{array} },e_2I_2={\small\begin{array}{c}1\\1\\2\end{array} }, E_1={\small\begin{array}{c}1\\1\end{array} }, E_2={\small\begin{array}{c}2.\end{array} }
  }\]
 All the nonzero indecomposable $\tau$-rigid $\Pi$-modules  are  $\{e_1\Pi, e_2\Pi, e_1I_1,e_2I_2,E_1,E_2\}.$

\end{example}


\begin{thebibliography}{AAAA}


\bibitem{AIR} T. Adachi, O. Iyama and I. Reiten, \emph{$\tau$-tilting theory}, Compos. Math. \textbf{150}(3), 415-452 (2014)

\bibitem{BB}A. Bj\"{o}rner, F. Brenti, \emph{Combinatorics of Coxeter groups}, Graduate Texts in Mathematics, vol. 231, Springer Berlin Heidelberg (2005)



\bibitem{BBK} S. Brenner, M. Butler and A. King, \emph{Periodic algebras which are almost Koszul}, Algebras and Representation Theory \textbf{5}, 331-367 (2002)
\bibitem{BIRS} A. B. Buan, O. Iyama, I. Reiten and J. Scott, \emph{Cluster structures for 2-Calabi-Yau categories and unipotent groups}, Compos. Math.
\textbf{145} (4), 1035-1079 (2009)

\bibitem{Ca} R. W. Carter, \emph{Lie algebras of finite and Affine Type}, 
Cambridge Studies in Advanced Mathematics, Cambridge University Press (2005)

\bibitem{CB} W. Crawley-Boevey, \emph{On the exceptional fibres of Kleinian singularities},  Amer. J. Math. \textbf{122}, 1027-1037 (2000)

\bibitem{GLS05} C. Gei\ss, B. Leclerc and J. Schr\"{o}er, \emph{Semicanonical bases and preprojective algebras},  Ann. Sc. \'{E}cole Norm. Sup. \textbf{38}, 193-253 (2005)

\bibitem{GLS06} C. Gei\ss, B. Leclerc and J. Schr\"{o}er, \emph{Rigid modules over preprojective algebras}, Inventiones Mathematicae \textbf{165}, 589-632 (2006)

\bibitem{GLS11} C. Gei\ss, B. Leclerc and J. Schr\"{o}er, \emph{Kac-Moody groups and cluster algebras}, Adv. Math. \textbf{228}, 329-433 (2011)

\bibitem{GLS14} C. Gei\ss, B. Leclerc and J. Schr\"{o}er, \emph{Quivers with relations for symmetrizable Cartan matrices I: Foundations}, Invent. Math. doi:10.1007/s00222-016-0705-1(2016)

\bibitem{GP} I. M. Gelfand and V. A. Ponomare, \emph{Model algebras and representations of graphs}, Funktsional. Anal. i Prilozhen. \textbf{13}, 1-12 (1979)

\bibitem{IR}O. Iyama and I. Reiten, \emph{Fomin-Zelevinsky mutation and tilting modules over Calabi-Yau algebras}, Amer. J. Math. \textbf{130}, 1089-1149 (2008)


\bibitem{L} G. Lusztig, \emph{Semicanonical bases arising from enveloping algebras}, Adv. Math. \textbf{151}, 129-139 (2000)





\bibitem{M} Y. Mizuno, \emph{Classifying $\tau$-tilting modules over preprojective algebras of Dynkin type}, Math. Zeit. \textbf{277}(3), 665-690 (2014)
\bibitem{N} H. Nakajima, \emph{Instantons on ALE spaces, quiver varieties, and Kac-Moody algebras}, Duke Math. \textbf{76}(2), 365-416 (1994)
\bibitem{Ric} J. Rickard, \emph{Morita theory for derived categories}, J. London Math. Soc. \textbf{29}(2), 436-456 (1989)

\bibitem{R} C. M. Ringel, \emph{The preprojective algebra of a quiver}, In: Algebras and modules II (Geiranger, 1966), 467-480, CMS Conf. Proc. 24, AMS (1998)
\bibitem{Yek} A. Yekutieli, \emph{Dualizing complexes, Morita equivalence and the derived Picard group of a ring}, J. London Math. Soc. \textbf{60}(2), 723-746 (1999)







\end{thebibliography}
\end{document}